\documentclass[12pt]{amsart}

\usepackage{a4wide}
\usepackage{amsmath, amssymb}
\usepackage{amsthm}
\usepackage[utf8]{inputenc}
\usepackage{subfigure}
\usepackage{enumerate}
\usepackage{cancel}
\usepackage{xspace}
\usepackage{float}
\usepackage{mathrsfs}
\usepackage[colorlinks, citecolor=blue, anchorcolor=blue, linkcolor=blue, urlcolor=blue]{hyperref}
\usepackage{color,xcolor}

\usepackage{wrapfig}

\usepackage{tikz}
\usetikzlibrary{patterns}

\usepackage{algorithm}
\usepackage{algorithmic}

\def\SYoung#1{\vbox{\smallskip\offinterlineskip
    \halign{&\vbox{##}\kern-\SThickness\cr #1}}}

\newdimen\SSquaresize \SSquaresize=4.5pt
\newdimen\SThickness \SThickness=.15pt
\newdimen\SCorrection \SCorrection=7pt

\def\SCarre#1{\hbox{\vrule width \SThickness
   \vbox to \SSquaresize{\hrule height \SThickness\vss
      \hbox to \SSquaresize{\hss$\scriptstyle#1$\hss}
   \vss\hrule height\SThickness}
   \unskip\vrule width \SThickness}
   \kern-\SThickness}

\allowdisplaybreaks

\makeatletter \@addtoreset{equation}{section}

\newtheorem{theorem}{Theorem}[section]

\newtheorem{lemma}[theorem]{Lemma}

\newcounter{mmacnt}
\def\restartmma{\setcounter{mmacnt}{0}}
\restartmma
\catcode`|=\active
\def|#1|{\mathrm{#1}}
\catcode`|=12
\def\mma@slash{/}
\catcode`/=\active
\makeatother\def\mma{@}\makeatletter\let\mma@at\mma\let\mma\undefined
\def/{\@ifnextchar/{\ifmmode\mathop{\;\mma@slash\!\mma@slash\;}\else\mma@slash\!\mma@slash\fi\mma@eat}{\mma@B}}
\def\mma@B{\@ifnextchar.{\ifmmode\mathop{\;\mma@slash\!.\;}\else\mma@slash\!.\fi\mma@eat}{\mma@C}}
\def\mma@C{\expandafter\@ifnextchar\mma@at{\ifmmode\mathop{\;\mma@slash\kern-.25ex @\;}\else\mma@slash\!@\fi\mma@eat}{\mma@D}}
\def\mma@D{\@ifnextchar;{\ifmmode\mathop{\;\mma@slash\!;\;}\else\mma@slash\!;\fi\mma@eat}{\mma@slash}}
\catcode`/=12
\def\mma@eat#1{}
\newenvironment{mma}{%
 \par\smallskip
 \catcode`|=\active\catcode`/=\active
 \parskip=0pt\parindent=0pt 
 \small
 \def\In##1\\{%
   \def\linebreak{\hfil\break\null\qquad}%
   \refstepcounter{mmacnt}%
   \hangindent=2.5em\hangafter=0
   \leavevmode
   \llap{\tiny\sffamily In[\arabic{mmacnt}]:=\kern.5em}%
   \mathversion{bold}$\displaystyle##1$
   \mathversion{normal}\par\kern-\lineskip
 }%
 \def\Print##1\\{%
   \def\linebreak{\hfil\break}%
   \hangindent=2.5em\hangafter=0
   \leavevmode ##1\par}%
 \def\Out{\@ifnextchar[{\mma@out}{\mma@out[@]}}%
 \def\mma@out[##1]##2\\{%
   \def\linebreak{\hfil\break\null}%
   \kern\abovedisplayskip\par
   \hangindent=2.5em\hangafter=0
   {\kern\lineskip\advance\lineskip2pt\kern-\lineskip
   \leavevmode
   \llap{\tiny\sffamily Out[\arabic{mmacnt}]=\kern.5em}%
   \ifx ##1@\else
     \rlap{\kern\hsize\kern-2.5em\llap{(##1)}}%
   \fi
   $\displaystyle##2$\hfil\null
   \par\kern-\lineskip}%
   \kern\lineskip
   \kern\belowdisplayskip
 }%
 \def\Warning##1##2\\{%
   \def\linebreak{\hfil\break}%
   \hangindent=2.5em\hangafter=0
   \leavevmode
   {\scriptsize##1 : ##2}\par}%
}{%
}

\newcommand{\refOut}[1]{{\small\sffamily Out[#1]}}

\def\MLabel#1{{\refstepcounter{mmacnt}\label{#1}}\addtocounter{mmacnt}{-1}}

\title[Asymptotic normality in Baxter permutations]{Asymptotic normality arising in Baxter permutations}

\author[James J.Y. Zhao]{James Jing Yu Zhao}
       \address{School of Accounting, Guangzhou College of Technology and Business,
       Foshan 528138, P. R. China.}
       \email{zhao@gzgs.edu.cn}



\begin{document}

\begin{abstract}
Baxter permutations arose in the study of fixed points of the composite of commuting functions by Glen Baxter in 1964. This type of permutations are counted by Baxter numbers $B_n$. It turns out that $B_n$ enumerate a lot of discrete objects such as the bases for subalgebras of the Malvenuto-Reutenauer Hopf algebra, the pairs of twin binary trees on $n$ nodes, or the diagonal rectangulations of an $n\times n$ grid. The refined Baxter number $D_{n,k}$ also count many interesting objects including the Baxter permutations of $n$ with $k-1$ descents and $n-k$ rises, twin pairs of binary trees with $k$ left leaves and $n-k+1$ right leaves, or plane bipolar orientations with $k+1$ faces and $n-k+2$ vertices.
In this paper, we obtain the asymptotic normality of the refined Baxter number $D_{n,k}$ by using a sufficient condition due to Bender. In the course of our proof, the computation involving $B_n$ and some related numbers is crucial, while $B_n$ has no closed form which make the computation untractable. To address this problem, we employ the method of asymptotics of the solutions of linear recurrence equations. Our proof is semi-automatic. All the asymptotic expansions and recurrence relations are proved by utilizing symbolic computation packages.
\end{abstract}

\maketitle

\section{Introduction}\label{Sec:1}

Baxter numbers enumerate Baxter permutations introduced by Glen Baxter \cite{Baxter1964} while studying fixed points of the composite of commuting functions in 1964. The formula of the Baxter numbers was obtained by Chung {\it et al.} \cite{CGHK1978} in 1978 as
\begin{equation}\label{def-Bn}
B_n = \sum_{k=1}^n \frac{2}{n(n+1)^2} \binom{n+1}{k-1} \binom{n+1}{k} \binom{n+1}{k+1}.
\end{equation}
For convenience, let $B_0=1$. The first few terms of $B_n$ are $1,\,1,\,2$, $6,\, 22,\, 92$, $422,\, 2074$, $10754,\, 58202$. See Sloane \cite[\textsc{A001181}]{oeis}.
A combinatorial proof of \eqref{def-Bn} was showed by Viennot \cite{Viennot1981}.
Mallows \cite{Mallows1979} further found that
\begin{equation}\label{eq:Dnk-of-Bn}
D_{n,k}=\frac{2}{n(n+1)^2}\binom{n+1}{k-1} \binom{n+1}{k} \binom{n+1}{k+1}
\end{equation}
is just the number of reduced Baxter permutations on $I_n$ that have exactly $k$ rises.
Moreover, Felsner {\it et al}. \cite{FFNO2011} introduced $\varTheta$-numbers as
$
\varTheta_{s,t}=\frac{2}{(s+1)^2(s+2)} \binom{s+t}{s} \binom{s+t+1}{s} \binom{s+t+2}{s}
$,
which count Baxter permutations with $s$ descents and $t$ rises.
Note that $D_{n,k}=\varTheta_{k-1,n-k}$. So the Baxter numbers can also be written as $B_n=\sum_{k=0}^{n-1} \varTheta_{k,n-k-1}$.

A lot of discrete and combinatorial objects have been found to be enumerated by the Baxter numbers $B_n$ or its refinement $D_{n,k}$. For example, $B_n$ enumerate the bases for subalgebras of the Malvenuto-Reutenauer Hopf algebra \cite{Reading2005}, the pairs of twin binary trees on $n$ nodes \cite{Giraudo2012}, the diagonal rectangulations of an $n\times n$ grid \cite{Law-Reading2012}, and the bases of Baxter-Cambrian Hopf algebra \cite{Chatel-Pilaud2017}. 
Besides, many bijections related to Baxter permutations and various objects counted by Baxter numbers were given, including alternating Baxter permutations \cite{CDV1986}, certain kinds of Young tableaux \cite{DG1996, CEF2005},
non-intersecting paths \cite{DG1998}, mosaic floorplans \cite{YCCG2003, AcBaPi2006},
plane bipolar orientations \cite{BBF2010}, Baxter families and related objects \cite{FFNO2011},
simple walks in Weyl chambers \cite{CFLM2018}, and Baxter tree-like tableaux \cite{ABBGS}.

The refined Baxter numbers $D_{n,k}$ enumerate not only Baxter permutations of $n$ with $k-1$ descents and $n-k$ rises \cite[Proposition 6.8]{FFNO2011}, but also twin pairs of binary trees with $k$ left leaves and $n-k+1$ right leaves \cite[Theorem 5.6]{FFNO2011}. Moreover, by Felsner {\it et al}. \cite[Proposition 5.7]{FFNO2011}, the numbers $D_{n,k}$ also count rectangulations of $X_{n-1}$ with $k-1$ horizontal and $n-k$ vertical segments, plane bipolar orientations with $k+1$ faces and $n-k+2$ vertices, and many other combinatorial objects.

Many combinatorial statistics were proved to be asymptotically normal, see \cite{Bender1973, Bona2009, CMW-2020, CYZ-2022, CWZ-2020,LWW-2023} for examples. This motivated us to investigate if the refined Baxter number satisfies this property.
The objective of this paper is to prove the asymptotic normality of the refined Baxter number $D_{n,k}$.

Suppose that $\{f_n(x)\}_{n\geq 0}$ is a sequence of univariate polynomials $f_n(x)\in \mathbb{R}[x]$ with nonnegative coefficients $a(n,k)$, that is,
\begin{align}\label{def-f}
f_n(x)=\sum_{k=0}^{n}a(n,k)x^k.
\end{align}
Let $X_n$ be a random variable for $n\geqslant 0$.
We say that the coefficient $a(n,k)$ is asymptotically normal with mean $\mu_n$ and variance $\sigma_n^2$ by a central limit theorem if
\begin{align}\label{eq-clan}
\lim\limits_{n\rightarrow \infty} \sup\limits_{x\in \mathbb{R}}
\left| \sum\limits_{k\leq \mu_n+x\sigma_n} p(n,k) -\frac{1}{\sqrt{2\pi}}\int_{-\infty}^x \exp({-t^2/2}) dt \right|=0,
\end{align}
where
$
p(n,k)={\rm Prob}(X_n=k)={a(n,k)}/{\sum_{j=0}^{n}a(n,j)}.
$
We say that $a(n,k)$ is asymptotically normal  with mean $\mu_n$ and variance $\sigma_n^2$  by a local limit theorem on the real set $\mathbb{R}$ if
\begin{align}\label{eq-local}
\lim\limits_{n\rightarrow \infty} \sup\limits_{x\in \mathbb{R}}
\left| \sigma_n p(n,\lfloor\mu_n+x \sigma_n\rfloor) -\frac{1}{\sqrt{2\pi}} \exp({-{x^2}/{2}}) \right|=0.
\end{align}
It is known that \eqref{eq-local} implies \eqref{eq-clan}. However, \eqref{eq-clan} does not imply \eqref{eq-local} in general. 
see Bender \cite{Bender1973} and Canfield \cite{Canfield}.

Bender \cite{Bender1973} established the following sufficient condition for determining if the coefficient of a sequence satisfies the asymptotic normality property. See also Harper \cite{Harper1967}.
\begin{theorem}\cite[Theorem 2]{Bender1973}\label{thm:Bender}
Suppose that $\{f_n(x)\}_{n\geq 0}$ is a sequence of real-rooted polynomial with nonnegative coefficients as in \eqref{def-f}.
Let
\begin{align}
\mu_n=\frac{f_n'(1)}{f_n(1)}\qquad \textrm{and} \qquad
\sigma_n^2=\frac{f_n''(1)}{f_n(1)}+\mu_n-\mu_n^2.\label{eq:mu-var}
\end{align}
If $\sigma_n^2\rightarrow +\infty$ as $n \rightarrow +\infty$, then the coefficients of $f_n(x)$ are asymptotically normal  with means $\mu_n$ and variances $\sigma_n^2$  by local and central limit theorems.
\end{theorem}

For $n\ge 2$ and $x\in \mathbb{R}$, define a kind of \emph{Baxter polynomial} by
\begin{equation}\label{def:PBnx-form1}
 \mathfrak{B}_n(x):
= \sum_{k=0}^n D_{n,k} x^{k}
= \sum_{k=0}^n \frac{2}{n(n+1)^2} \binom{n+1}{k-1} \binom{n+1}{k} \binom{n+1}{k+1} x^{k}.
\end{equation}

The main result of this paper is as follows.

\begin{theorem}\label{Thm:asy-normal}
Let $\mathfrak{B}_n(x)$ be defined in \eqref{def:PBnx-form1}. Let $\mu_n=\frac{\mathfrak{B}_n'(1)}{\mathfrak{B}_n(1)}$ and $\sigma_n^2=\frac{\mathfrak{B}_n''(1)}{\mathfrak{B}_n(1)}+\mu_n-\mu_n^2$.
Then the coefficient of $\mathfrak{B}_n(x)$, that is, the number $D_{n,k}$ 
is asymptotically normal by local and central limits theorems with
\begin{align}
 \mu_n \sim \frac{n}{2} \qquad {\rm and}\qquad 
 \sigma_n^2 \sim  \frac{n}{12}.\label{sim:mu_n&sigms_n^2}
\end{align}
\end{theorem}

We shall apply Theorem \ref{thm:Bender} to prove our main result.
It should be mentioned that the real-rootedness of $\mathfrak{B}_n(x)$ was obtained by Yang \cite{Yang} by using multiplier-sequences. See also King, Rota and Yan \cite[\S 6.4]{KRY2019} or Br\"{a}nd\'{e}n \cite{Branden-2006} for information of multiplier-sequences.

\begin{theorem}\cite{Yang}\label{thm:real-rt-PB}
The Baxter polynomial $\mathfrak{B}_n(x)$ has only real zeros for each $n\geqslant 2$.
\end{theorem}

So, in order to prove Theorem \ref{Thm:asy-normal}, we need only to show \eqref{sim:mu_n&sigms_n^2} where the computation of ${\mathfrak{B}_n'(1)}/{\mathfrak{B}_n(1)}$ and ${\mathfrak{B}_n''(1)}/{\mathfrak{B}_n(1)}$ are key ingredient. Note that $\mathfrak{B}_n(1)=B_n$ has no closed form which make the proof difficult. To deal with this problem, we utilize the tools in asymptotics of the solutions of linear recurrence equations.

The remainder of this paper is organized as follows.
To be self-contained, we first present another proof of Theorem \ref{thm:real-rt-PB}, in Section \ref{sec:realzeors}, by using a result of Driver, Jordaan, and Mart\'{i}nez-Finkelshtein \cite{DJM-F2007}. In Section \ref{sec:Recs}, we prove some recurrence relations involving the Baxter polynomial and its derivatives by applying packages of Zeilberger \cite{Zeilberger1991} and Koutschan \cite{Koutschan2009}. These recurrences will be adopted in our proof of the main result. In Section \ref{sec:prf}, with the help of Theorem \ref{thm:Bender}, we complete our  semi-automatic proof of Theorem \ref{Thm:asy-normal} by employing the tools of asymptotic analysis and a package due to Kauers \cite{Kauers2011}.

\section{Another proof of the Real-rootedness of Baxter polynomial}\label{sec:realzeors}

To be self-contained, we shall give another proof of Theorem \ref{thm:real-rt-PB} in this section. 
Recall that the classical hypergeometric functions ${_3}F_2$ \cite{Rainville1960} is defined by
\begin{eqnarray}\label{eq:3F2}
{_3F_2}\hskip -3pt\left[\hskip -3pt
\begin{array}{@{}r@{}l@{}r@{}}
\begin{array}{c}
p_1,\ p_2,\ p_3\\
q_1,\ q_2
\end{array}
&\hskip -3pt ;&\; {\displaystyle z}
\end{array}
\right]
=\sum_{k=0}^{\infty} \frac{(p_1)_n (p_2)_n (p_3)_n}{(q_1)_n (q_2)_n\ n!}z^n,
\end{eqnarray}
where $(a)_n=a(a+1)\cdots (a+n-1)$ for $n\geq 1$ is the Pochhammer symbol with $(a)_0=1$.
By combining Laguerre’s theorem, Hadamard's factorization theorem, Laplace transformation of a P\'{o}lya frequency function, and a theorem of Schur,
Driver {\it et al.} \cite{DJM-F2007} derived the following result on $_3F_2$ hypergeometric polynomials.
\begin{theorem}\cite[Theorem 9]{DJM-F2007}\label{thm:DJMF}
The polynomial \eqref{eq:3F2} has only positive real zeros if $q_1,q_2>0$ and $p_1,p_2,p_3$ are non-positive integers.
\end{theorem}

We are now ready to give another proof of Theorem \ref{thm:real-rt-PB}.
\begin{proof}[Proof of Theorem \ref{thm:real-rt-PB}]
Clearly, the constant term of $\mathfrak{B}_n(x)$ is zero by the convention that $\binom{n}{-1}=0$.
Since $D_{n,k}=\varTheta_{k-1,n-k}$ as described in Section \ref{Sec:1},
$
\mathfrak{B}_n(x)
=\sum_{k=1}^n D_{n,k}x^{k}
= \sum_{k=1}^{n} \varTheta_{k-1,n-k}\, x^{k} = x \sum_{k=0}^{n-1} \varTheta_{k,n-k-1}\, x^k.
$
That is,
\begin{align}\label{eq:PBnx-form2}
\mathfrak{B}_n(x)
&= x \sum_{k=0}^{\infty} \frac{2}{(k+1)^2(k+2)}{\binom{n-1}{k}} \binom{n}{k} {\binom{n+1}{k}}x^{k}.
\end{align}
Apply the relation
$\binom{n}{k}=\frac{n(n-1)\cdots (n-k+1)}{k!}=\frac{(-1)^k(-n)_k}{k!}$
to the three binomials coefficients on the right-hand side of \eqref{eq:PBnx-form2}. It follows that for $n\geq 1$,
\begin{align}
  \frac{\mathfrak{B}_n(x)}{x}
&= \sum_{k=0}^{\infty} \frac{(-n)_k (-n-1)_k (-n+1)_k}{(2)_k (3)_k\, k!}\, (-x)^{k}
= {_3F_2}\hskip -3pt\left[\hskip -3pt
\begin{array}{@{}r@{}l@{}r@{}}
\begin{array}{c}
-n,\ -n-1,\ -n+1\\
2,\qquad 3
\end{array}
&\hskip -3pt ;&\; {\displaystyle -x}
\end{array}
\right].\label{eq:PBn-3F2}
\end{align}

Notice that the right-hand side of \eqref{eq:PBn-3F2} is in the form of \eqref{eq:3F2} with
$p_1=-n$, $p_2=-n-1$, $p_3=-n+1$ which are all non-positive integers, and $q_1=2,\,q_2=3$ which are positive integers, for any $n\geqslant 1$. Hence by Theorem \ref{thm:DJMF}, the Baxter polynomial $\mathfrak{B}_n(x)$ has only real and nonpositive zeros for any $n\geqslant 2$. This completes the proof.
\end{proof}

\section{Recurrence relations}\label{sec:Recs}
Before showing the proof of Theorem \ref{Thm:asy-normal}, we first recall and give some recurrence relations related to the Baxter numbers $B_n$, the Baxter polynomial $\mathfrak{B}_n(x)$ and its derivatives. These recurrences will be applied to find asymptotic formulae of $\mathfrak{B}_n(1)$, $\mathfrak{B}_n'(1)$ and $\mathfrak{B}_n''(1)$. As will be seen, this part is a key ingredient in our proof of Theorem \ref{Thm:asy-normal}.

\begin{lemma}[Ollerton]\label{lem:rec-Bn}
Let $B_n$ be defined by \eqref{def-Bn}. Then we have for $n\geqslant 3$,
\begin{align}\label{eq-rec-Bn}
 (n+2)(n+3)B_{n}=(7n^2+7n-2)B_{n-1}+8(n-1)(n-2)B_{n-2}.
\end{align}
\end{lemma}

Notice that the recurrence \eqref{eq-rec-Bn} is due to Richard L. Ollerton as was recorded at OEIS \cite{oeis}.
Since $B_n$ is a sum of hypergeometric terms on $k$, it is easy to prove \eqref{eq-rec-Bn} with Zeilberger's algorithm \cite{Zeilberger1991}.

\begin{proof}
Let $N$ be a shift operator such that $N f_n=f_{n+1}$. Running the {\tt Maple} command ${\tt ct}(D_{n,k},3,k,n,N)$ outputs
$Op(n), R(n,k)$, where $Op(n)=-8 (n+1) n+(-7 n^2-35 n-40) N+(n+5) (n+4) N^2$ is an operator and $R(n,k)$ is a rational function of $n$ and $k$. This means that $Op(n) B_n=0$, which leads to \eqref{eq-rec-Bn}.
\end{proof}

\begin{lemma}\label{lem:rec-Bn'}
Let $\mathfrak{B}_{n}(x)$ be defined by \eqref{def:PBnx-form1}. Then for $n\geqslant 4$, we have
\begin{align}\label{eq:rec-PBn'}
 \dot{a}_n \mathfrak{B}_{n}'(1)
= \dot{b}_n \mathfrak{B}_{n-1}'(1) + \dot{c}_n \mathfrak{B}_{n-2}'(1)
   + \dot{d}_n \mathfrak{B}_{n-3}'(1),
\end{align}
where
\begin{equation*}
\begin{aligned}
\dot{a}_n& = (3n-2)(n+3)(n+2)(n+1),\\
\dot{b}_n& = 6(n+1)(3n^3+5n^2-4),\\
\dot{c}_n& = 3(n-2)(15n^3+6n^2-7n-6),\\
\dot{d}_n& = 8(3n+1)(n-3)(n-2)(n-1),
\end{aligned}
\end{equation*}
and
\begin{align}\label{eq:rec-PBn''}
  \ddot{a}_n \mathfrak{B}_{n}''(1)
= \ddot{b}_n \mathfrak{B}_{n-1}''(1) + \ddot{c}_n \mathfrak{B}_{n-2}''(1)
   + \ddot{d}_n \mathfrak{B}_{n-3}''(1),
\end{align}
where
\begin{align*}
\ddot{a}_n& = (9n-7)(n+3)(n+2)(n-2),\\
\ddot{b}_n& = 6(9n^4+5n^3-12n^2-20n+4),\\
\ddot{c}_n& = 3(n-1)(45n^3-77n^2-50n+8),\\
\ddot{d}_n& = 8(9n+2)(n-1)(n-2)(n-3).
\end{align*}
\end{lemma}

\begin{proof}
We first prove \eqref{eq:rec-PBn'}, the recurrence for $\mathfrak{B}_n'(1)$. By \eqref{def:PBnx-form1}, we have that
\begin{align*}
  \mathfrak{B}_n'(x)
 =\sum_{k=0}^n k D_{n,k} x^{k-1}
 =\sum_{k=0}^n \frac{2k}{n(n+1)^2} \binom{n+1}{k-1} \binom{n+1}{k} \binom{n+1}{k+1} x^{k-1}.
\end{align*}
Clearly, $\mathfrak{B}_n'(1)$ is a sum of hypergeometric terms on $k$. Thus it will be more concise to show a symbolic proof. In order to do so, set $h=k\cdot D_{n,k}$ and perform the {\tt Maple} command ${\tt ct}(h,3,k,n,N)$. By checking the outputs, we arrive at the recurrence relation \eqref{eq:rec-PBn'}. The details is similar to the argument described in the proof of Lemma \ref{lem:rec-Bn}, and hence is
omitted here.
Further, note that
\begin{align*}
  \mathfrak{B}_n''(x)=\sum_{k=0}^n k(k-1) D_{n,k} x^{k-2}.
\end{align*}
Setting $h=k(k-1)D_{n,k}$ and running the command ${\tt ct}(h,3,k,n,N)$, one can obtain the recurrence relation \eqref{eq:rec-PBn''}  in the same manner.
This completes the proof.
\end{proof}

\begin{lemma}\label{lem:rec-fn''2}
Let $\mathfrak{B}_{n}(x)$ be defined by \eqref{def:PBnx-form1}. Then for $n\ge 2$, we have
\begin{align}
  \mathfrak{B}_{n+1}(x)
=&\  a_{n}(x) \mathfrak{B}_{n}(x)
   +b_{n}(x) \mathfrak{B}_{n}'(x)
   +c_{n}(x) \mathfrak{B}_{n}''(x),\label{rec:PBnx3}\\
  \mathfrak{B}_{n}(x)
=&\  \widetilde{a}_{n}(x) \mathfrak{B}_{n+1}'(x)
   +\widetilde{b}_{n}(x) \mathfrak{B}_{n}'(x)
   +\widetilde{c}_{n}(x) \mathfrak{B}_{n}''(x),\label{rec:PBnx4}
\end{align}
where
{\scriptsize
\begin{align*}
&a_{n}(x)=\frac{(10n^2+25n+12)x+n^2+4n}{(n+3)(n+4)},\
 b_{n}(x)=-\frac{3x(5nx+4x-n-4)}{(n+3)(n+4)},\
 c_{n}(x)=\frac{6x^2(x+1)}{(n+3)(n+4)},\\
&\widetilde{a}_{n}(x)=\frac{(n+3)(n+4)}{6n^3+28n^2+37n+12},\
 \widetilde{b}_{n}(x)=\frac{(8n^2+23n+12)x-(n+3)(n+4)}{6n^3+28n^2+37n+12},\
 \widetilde{c}_{n}(x)=-\frac{3(n+2)x(x+1)}{6n^3+28n^2+37n+12}.
\end{align*}
}%
\end{lemma}

\begin{proof}
We shall apply the {\tt HolonomicFunctions}\footnote{See https://www3.risc.jku.at/research/combinat/software/ergosum/RISC/HolonomicFunctions.html.} for {\tt Mathematica} given by Koutschan \cite{Koutschan2009} to prove \eqref{rec:PBnx3} and \eqref{rec:PBnx4}. Let us first load this package.
\begin{mma}
\In  << |RISC'HolonomicFunctions'| \\
\Print HolonomicFunctions Package version 1.7.3 (21-Mar-2017)\\
\Print written by Christoph Koutschan\\
\Print Copyright Research Institute for Symbolic Computation (RISC),\\
\Print Johannes Kepler University, Linz, Austria\\
\end{mma}%

\noindent Then run the command Annihilator[$expr$, $ops$] which computes annihilating relations for $expr$ with respect to the Ore operator(s) $ops$. In the following input, $expr=\mathfrak{B}_n(x)$ and $ops$ consists of S$[n]$ and Der$[x]$, where S$[n]$ stands for a shift operator such that S$[n]\mathfrak{B}_n(x)=\mathfrak{B}_{n+1}(x)$ and Der$[x]$ denotes operator $\partial/\partial x$.
\begin{mma}\MLabel{MMA:9}
\In |ann = Annihilator|[|Sum|[2/(n (n+1)^2) |Binomial|[n+1, k-1]
        |Binomial|[n+1,k]\linebreak
     |Binomial|[n+1, k+1] x^{k}, \{k, 0, n\}], \{|S|[n], |Der|[x]\}]\\

\Out \{(-6x^2-6x^3)\!D_{\!x}^2+(12+7n+n^2)S_n+(-12 x - 3 n x + 12 x^2 + 15 n x^2)D_{\!x}\linebreak +(-4 n - n^2 - 12 x - 25 n x - 10 n^2 x),\linebreak
     (-6 x - 2 n x) S_n D_{\!x} + (6 + 5 n + n^2) S_n + (-n x - n x^2) D_{\!x} + (-2 n - n^2 + 3 n x + 2 n^2 x),\linebreak
     (120 + 94 n + 24 n^2 + 2 n^3) S_n^2 + (-120 - 145 n - 56 n^2 - 7 n^3 - 120 x - 145 n x - 56 n^2 x - 7 n^3 x) S_n + (21 n x + 9 n^2 x - 21 n x^3 - 9 n^2 x^3) D_{\!x} + (30 n + 23 n^2 + 5 n^3 - 129 n x - 140 n^2 x - 35 n^3 x + 51 n x^2 + 53 n^2 x^2 + 14 n^3 x^2)\}\\
\end{mma}

The output \refOut{\ref{MMA:9}} gives a list of Ore Polynomial expressions which form a Gr\"{o}bner basis of an annihilating left ideal for $f_n(x)$, where $S_n={\rm S}[n]$ and $D_{\!x}={\rm Der}[x]$. See Ore \cite{Ore1933} and Koutschan \cite{Koutschan2009} for more information of the Ore algebra and Ore polynomial.

In order to prove recurrence \eqref{rec:PBnx3}, we employ the command OreReduce[$opoly$, $ann$] which reduces the Ore polynomial $opoly$ modulo the Gr\"{o}bner basis $ann$.
In the following input, $opoly$ corresponds to the Ore polynomial of recurrence \eqref{rec:PBnx3} and $ann$ is given by \refOut{\ref{MMA:9}}.
\begin{mma}
\In |OreReduce|[(n+3)(n+4)|S|[n]-((10n^2+25n+12)x+n^2+4n)\linebreak
      +3x(5nx+4x-n-4)|Der|[x]-6x^2(x+1)|Der|[x]^2, |ann|]\\

\Out 0\\
\end{mma}

The return value is zero, which means that $opoly$ is contained in the left ideal generated by the elements of $ann$, and hence $opoly (\mathfrak{B}_{n}(x))=0$, that is, the recurrence \eqref{rec:PBnx3} is valid.

To prove \eqref{rec:PBnx4}, set $opoly$ as follows and run the command OreReduce[$opoly$, $ann$].
\begin{mma}\MLabel{MMA:11}
\In |OreReduce|[6n^3+28n^2+37n+12-(n+3)(n+4)|S|[n]|Der|[x]\linebreak -((8n^2+23n+12)x-(n+3)(n+4))|Der|[x]+3(n+2)x(x+1)|Der|[x]^2, |ann|]\\

\Out 0\\
\end{mma}
This means that \eqref{rec:PBnx4} holds true. Thus, Lemma \ref{lem:rec-fn''2} is proved.
\end{proof}

We remark that the recurrence relations related to the Baxter polynomial presented in this paper can also be proved by using the extended Zeilberger algorithm developed by Chen, Hou and Mu \cite{CHM-2012}.

\section{Proof of Theorem \ref{Thm:asy-normal}}\label{sec:prf}

The objective of this section is to complete the proof of Theorem \ref{Thm:asy-normal} that $D_{n,k}$ is asymptotically normal by local and central limits theorems, by applying Theorem \ref{thm:Bender} of Bender. For this purpose, we need to evaluate the variance $\sigma_n^2$ and its limitation at $+\infty$, where the expressions of $\mathfrak{B}_n'(1)/\mathfrak{B}_n(1)$ and $\mathfrak{B}_n''(1)/\mathfrak{B}_n(1)$ are crucial. Note that $\mathfrak{B}_n(1)=B_n$ has no closed form, thus we shall employ some tools in asymptotics of the solutions of linear recurrence equations.

Let us first have a brief overview of some concepts and results in this theory.
The theory of asymptotics of the solutions of linear recurrence equations was pioneered by Poincar\'{e} \cite{Poincare1885}, and was developed by Birkhoff \cite{Birkhoff1930}, Birkhoff and Trjitzinsky \cite{Birkhoff-Trjitzinsky1933}, Wimp and Zeilberger \cite{Wimp-Zeilberger1985}, and Hou and Zhang \cite{Hou-Zhang2019}.
A sequence $\{f_n\}_{n\geqslant 0}$ is called \emph{P-recursive} if it satisfies a homogeneous linear recurrence of finite degree with polynomial coefficients, see Stanley \cite[Section 6.4]{Stanley}.
Equivalently,
\begin{align}\label{eq:P-r-seq}
 f_{n}=r_1(n) f_{n-1} + r_2(n) f_{n-2} + \cdots + r_{\ell}(n) f_{n-\ell}
\end{align}
where $r_i(n)$ are rational functions of $n$. 
It was showed \cite{Birkhoff-Trjitzinsky1933, Wimp-Zeilberger1985} that there exists a formal solution of \eqref{eq:P-r-seq} which is asymptotically equal to a linear combination of terms of the form
\begin{align}\label{eq:B-T}
 e^{Q(\rho,n)}\, s(\rho,n),
\end{align}
where
\begin{align}
 Q(\rho,n) & = \mu_0 n \log n + \sum_{j=1}^{\rho} \mu_j n^{j/\rho},\label{eq:Qrhon} \\
 s(\rho,n) & = n^{\nu} \sum_{j=0}^{t-1}(\log n)^j \sum_{s=0}^{M-1} b_{sj} n^{-s/\rho},\label{eq:srhon}
\end{align}
with $\rho,\, t,\, M$ being positive integers and $\mu_j,\, \nu,\, b_{sj}$ being complex numbers.

For instance, it is easy to verify that the sequence $\mathfrak{u}(n)=\sum_{k=0}^n \binom{n}{k}^3$ satisfies the recurrence
$$
(n+2)^2 \mathfrak{u}(n+2)-(7n^2+21n+16) \mathfrak{u}(n+1)-8(n+1)^2 \mathfrak{u}(n)=0.
$$
Wimp and Zeilberger \cite[Example 3.3]{Wimp-Zeilberger1985} showed that the Birkhoff method gives
\begin{align}\label{P-type-asyapp}
    \mathfrak{u}(n)
\sim K \frac{8^n}{n}\left(1+\frac{c_1}{n}+o\left(\frac{1}{n}\right)\right)
\end{align}
for some $K>0$.
As mentioned by Wimp and Zeilberger \cite[P. 175]{Wimp-Zeilberger1985}, the Birkhoff method does not gives the exact value of $K$. However, we only need that $K>0$ in our proof. For more details of the theory of asymptotics of the solutions of linear recurrence equations, please see \cite{Wimp-Zeilberger1985, Hou-Zhang2019} and the literature cited therein.

The procedure of computing $Q(\rho,n)$ and $s(\rho,n)$ from the recurrence relation of $f_n$ have been implemented into packages, such as {\tt AsyRec}\footnote{See \href{https://sites.math.rutgers.edu/~zeilberg/tokhniot/AsyRec.txt}
{https://sites.math.rutgers.edu/\url{~}zeilberg/tokhniot/AsyRec.txt}.} given by Zeilberger \cite{Zeilberger2016}, {\tt asymptotics.m}\footnote{See 
\href{http://www.kauers.de/software.html}{http://www.kauers.de/software.html}.} given by Kauers \cite{Kauers2011}, and {\tt P-rec.m} given by Hou and Zhang \cite{Hou-Zhang2019}.
By using any one of these packages, one can obtain the asymptotic expansion of a \emph{P}-recursive sequence.
The resulted asymptotic expansion is usually in the form of the right-hand side of \eqref{P-type-asyapp} without the positive constant $K$, which is called a \emph{Puiseux-type approximation}. Let $\{f_n\}_{n\geqslant 0}$ be a sequence of real numbers. Suppose that there exist real numbers $\ell_i,\,\alpha_i$ with
$
\alpha_0<\alpha_1<\cdots<\alpha_m
$
such that
$$
\lim_{n\rightarrow \infty} n^{\alpha_m}\left(f_n-\sum_{i=0}^m \frac{\ell_i}{n^{\alpha_i}}\right)=0.
$$
The summation
$$
g_n=\sum_{i=0}^m \frac{\ell_i}{n^{\alpha_i}}
$$
is called a \emph{Puiseux-type approximation} of $f_n$ and denoted by $f_n \approx g_n$.
Thus $f_n$ can be written with the standard little-o notation as
$$
f_n=\frac{\ell_0}{n^{\alpha_0}}+\frac{\ell_1}{n^{\alpha_1}}+\cdots
  +o\left(\frac{1}{n^{\alpha_m}}\right).
$$

In order to prove Theorem \ref{Thm:asy-normal}, we need the following Puiseux-type approximations.

\begin{lemma}\label{lem:P-app-fn'1''}
Let $f_n(x):=\mathfrak{B}_n(x)$. Then we have
\begin{align}
   f_n(1)
 =&\, C_1\cdot \frac{8^n}{n^{4}}
    \left(1-\frac{22}{3n}+\frac{955}{27 n^2}
     +o\left(\frac{1}{n^2}\right)\right),\label{eq:asy-fn1}\\
   f_n'(1)
 =&\, C_2\cdot \frac{8^n}{n^{3}}\left(1-\frac{19}{3n}
    +\frac{757}{27 n^2}
     +o\left(\frac{1}{n^2}\right)\right),\label{eq:asy-fn'1}\\
   f_n''(1)
 =&\, C_3\cdot \frac{8^n}{n^{2}}\left(1-\frac{7}{n}
    +\frac{877}{27 n^2}+o\left(\frac{1}{n^2}\right)\right),\label{eq:asy-fn''1}
\end{align}
where $C_1$, $C_2$ and $C_3$ are certain positive constants.
\end{lemma}

\begin{proof}
The asymptotic expansions given in \eqref{eq:asy-fn1}, \eqref{eq:asy-fn'1} and \eqref{eq:asy-fn''1} can be proved with {\tt Mathematica} package, such as {\tt AsyRec} given by Zeilberger \cite{Zeilberger2016}, {\tt asymptotics.m} given by Kauers \cite{Kauers2011}, or {\tt P-rec.m} given by Hou and Zhang \cite{Hou-Zhang2019}.

We first show the asymptotic expansion \eqref{eq:asy-fn1}. For example, to use the {\tt Mathematica} package {\tt asymptotics.m}, let us first import the package.

\begin{mma}
\In  << |RISC'Asymptotics'| \\
\Print Asymptotics Package version 0.3\\
\Print written by Manuel Kauers\\
\Print Copyright Research Institute for Symbolic Computation (RISC),\\
\Print Johannes Kepler University, Linz, Austria\\
\end{mma}%
\noindent Then we run the command {\tt Asymptotics}[$L, f[n]$, Order $\rightarrow d$] where $f[n]=B_n$ and $L=0$ is the recurrence of $f(n)$ provided by Lemma \ref{lem:rec-Bn} and $d$ can be any positive integer.
\begin{mma}
\In |L| = (n + 2) (n + 3) f[n] - (7 n^2 + 7 n - 2) f[n - 1]-8(n - 1) (n - 2) f[n - 2];\\

\In |Asymptotics[L,|\, f[n],\, |Order| \rightarrow 2]\\

\Out \left\{\frac{(-1)^n \left(1 + \frac{415}{27 n^2}
       - \frac{14}{3 n}\right)}{n^4},\
     \frac{8^n \left(1 + \frac{955}{27 n^2}
       - \frac{22}{3 n}\right)}{n^4}\right\}\\
\end{mma}

The output suggests that there are two formal solutions to the recursion \eqref{eq-rec-Bn},
$$
(-1)^n n^{-4} \sum_{i=0}^{\infty} \alpha_i n^{-i}
\qquad {\rm and} \qquad
8^n n^{-4} \sum_{i=0}^{\infty} \beta_i n^{-i},
$$
where $\alpha_i, \beta_i$ are real numbers. Since the first solution tends to zero as $n$ tends to infinity, it follows that $f_n(1)$ has a Puiseux-type approximation of the form
$$
f_n(1) \approx 8^n n^{-4} \sum_{i=0}^{d} \beta_i n^{-i},
$$
for any positive integer $d$. By setting $d=2$, the package gives \eqref{eq:asy-fn1} for a certain constant $C_1>0$.

For \eqref{eq:asy-fn'1} and \eqref{eq:asy-fn''1}, we input the following commands where $g[n]=\mathfrak{B}_{n}'(1)$, $h[n]=\mathfrak{B}_{n}''(1)$, $L2=0$ and $L3=0$ are the recursions \eqref{eq:rec-PBn'} and \eqref{eq:rec-PBn''}, respectively.
\begin{mma}
\In |L2| = (3 n - 2) (n + 3) (n + 2) (n + 1) g[n] -
 6 (n + 1) (3 n^3 + 5 n^2 - 4) g[n - 1]\linebreak -
 3 (n - 2) (15 n^3 + 6 n^2 - 7 n - 6) g[n - 2] -
 8 (3 n + 1) (n - 3) (n - 2) (n - 1) g[n - 3];\\
\end{mma}
\begin{mma}\MLabel{MMA:5}
\In |Asymptotics[L2|,\, g[n],\, |Order| \rightarrow 2]\\

\Out \Big\{\frac{(-1)^n \left(1 + \frac{575}{27 n^2}
       - \frac{50}{9 n}\right)}{n^4},\
     \frac{(-1)^n \left(1 
       - \frac{87}{50 n}\right)}{n^3},\
     \frac{8^n \left(1 
       + \frac{757}{27 n^2}
       - \frac{19}{3 n}\right)}{n^3}\Big\}\\

\In |L3| = (9 n - 7) (n + 3) (n + 2) (n - 2) h[n]
   -6 (9 n^4 + 5 n^3 - 12 n^2 - 20 n + 4) h[n - 1]\linebreak
   -3 (n - 1) (45 n^3 - 77 n^2 - 50 n + 8) h[n - 2]
   -8 (9 n + 2) (n - 1) (n - 2) (n - 3) h[n - 3];\\
\end{mma}
\begin{mma}\MLabel{MMA:7}
\In |Asymptotics[L3|,\, h[n],\, |Order| \rightarrow 2]\\

\Out \Big\{\frac{(-1)^n \left(1 
       + \frac{295}{27 n^2}
       - \frac{34}{9 n}\right)}{n^3},\
     \frac{(-1)^n \left(1 
       - \frac{39}{34 n}\right)}{n^2},\
     \frac{8^n \left(1 
       + \frac{877}{27 n^2}
       - \frac{7}{n}\right)}{n^2}\Big\}\\
\end{mma}
In a similar arguments, one can obtain the asymptotic expansions \eqref{eq:asy-fn'1} and \eqref{eq:asy-fn''1} from \refOut{\ref{MMA:5}} and \refOut{\ref{MMA:7}}, respectively, with certain constants $C_2>0$ and $C_3>0$.
Hence, Lemma \ref{lem:P-app-fn'1''} is proved.
\end{proof}

The following lemma is also necessary to our proof of Theorem \ref{Thm:asy-normal}.

\begin{lemma}\label{lme:4.2}
Let $f_n(x):=\mathfrak{B}_n(x)$. Then we have
\begin{align}\label{eq:lim-fn1n&fn1'n'}
 \lim_{n\rightarrow +\infty}\frac{f_{n+1}(1)}{f_n(1)}=\lim_{n\rightarrow +\infty}\frac{f_{n+1}'(1)}{f_n'(1)}=8.
\end{align}
\end{lemma}

\begin{proof}
Let $A:=\lim\limits_{n\rightarrow +\infty}\frac{f_{n+1}(1)}{f_n(1)}$.
Notice that $B_n=f_n(1)$. It is easy to derive from \eqref{eq-rec-Bn} that
$
 A^2-7A-8=0
$
whose zeros are $A_1=-1$ and $A_2=8$. Since $A\ge 0$ by \eqref{def:PBnx-form1}, we have $A=\lim\limits_{n\rightarrow +\infty}\frac{f_{n+1}(1)}{f_n(1)}=8$.

Let $\hat{A}:=\lim\limits_{n\rightarrow +\infty}\frac{f_{n+1}'(1)}{f_n'(1)}$.
To obtain the value of $\hat{A}$, first replace $n$ by $n+1$ in \eqref{eq:rec-PBn'}, we get
\begin{align*}
  \dot{a}_{n+1} f_{n+1}'(1)
= \dot{b}_{n+1} f_{n}'(1) + \dot{c}_{n+1} f_{n-1}'(1) + \dot{d}_{n+1} f_{n-2}'(1).
\end{align*}
Dividing by $\dot{b}_{n+1} f_n'(1)$ yields
\begin{align*}
  \frac{\dot{a}_{n+1}}{\dot{b}_{n+1}} \frac{f_{n+1}'(1)}{f_{n}'(1)}
= 1 + \frac{\dot{c}_{n+1}}{\dot{b}_{n+1}} \frac{f_{n-1}'(1)}{f_{n}'(1)}
  + \frac{\dot{d}_{n+1}}{\dot{b}_{n+1}} \frac{f_{n-2}'(1)}{f_{n}'(1)}.
\end{align*}
Letting $n\rightarrow +\infty$ over each side of the above equation, after some collection, we obtain the following cubic equation
\begin{align*}
  \frac{1}{6}(\hat{A}-8)(\hat{A}+1)^2=0.
\end{align*}
Because $\hat{A}\ge 0$ by \eqref{def:PBnx-form1}, we have $\hat{A}=\lim\limits_{n\rightarrow +\infty}\frac{f_{n+1}'(1)}{f_n'(1)}=8$.
This completes the proof.
\end{proof}

We are now in a position to give a proof of Theorem \ref{Thm:asy-normal}.

\begin{proof}[Proof of Theorem \ref{Thm:asy-normal}]
Let $f_n(x):=\mathfrak{B}_n(x)$. Let $\mu_n=\frac{f_n'(1)}{f_n(1)}$ and $\sigma_n^2=\frac{f_n''(1)}{f_n(1)}+\mu_n-\mu_n^2$.
By Theorems \ref{thm:Bender} and \ref{thm:real-rt-PB}, it is sufficient to prove \eqref{sim:mu_n&sigms_n^2}.
To this end, we first compute the asymptotic expansions of the expectation $\mu$ and the variances $\sigma_n^2$. By \eqref{eq:asy-fn1} and \eqref{eq:asy-fn'1}, we have
\begin{align*}
 \mu_n  =  \frac{f_n'(1)}{f_n(1)}
 & = \frac{C_2}{C_1}\, n \left(1-\frac{19}{3n}+\frac{955}{27 n^2}
   +o\left(\frac{1}{n^2}\right)\right)\left/
  \left(1-\frac{22}{3n}+\frac{955}{27 n^2}+o\left(\frac{1}{n^2}\right)\right)\right.\\
 & = \frac{C_2}{C_1}\, n \left(1-\frac{19}{3n}+\frac{955}{27 n^2}
   +o\left(\frac{1}{n^2}\right)\right)
  \left(1+\frac{22}{3n}+\frac{497}{27 n^2}+o\left(\frac{1}{n^2}\right)\right)\\
 &= \frac{C_2}{C_1}\, n \left( 1+\frac{1}{n} + \frac{22}{3n^2}+ o\left(\frac{1}{n^2}\right)\right).
\end{align*}
By \eqref{eq:asy-fn1} and \eqref{eq:asy-fn''1}, we see that
\begin{align*}
 \frac{f_n''(1)}{f_n(1)}
 & = \frac{C_3}{C_1}\, n^2 \left(1-\frac{7}{n}+\frac{877}{27 n^2}
     +o\left(\frac{1}{n^2}\right)\right)
      \left(1+\frac{22}{3n}+\frac{497}{27 n^2}+o\left(\frac{1}{n^2}\right)\right)\\
 & = \frac{C_3}{C_1}\, n^2 \left( 1+\frac{1}{3n} - \frac{4}{9n^2}
     + o\left(\frac{1}{n^2}\right)\right),
\end{align*}
and hence
\begin{align}
   \sigma_n^2
 =&\, \frac{f_n''(1)}{f_n(1)}+\mu_n-\mu_n^2\nonumber\\
 =&\, \frac{C_3}{C_1}\, n^2 \left(  1+\frac{1}{3n} - \frac{4}{9n^2}
  + o\left(\frac{1}{n^2}\right)\right)
  + \frac{C_2}{C_1}\, n \left( 1+\frac{1}{n} + \frac{22}{3n^2} + o\left(\frac{1}{n^2}\right)\right)\nonumber\\
  &\qquad - \left(\frac{C_2}{C_1}\right)^2 n^2 \left(1+\frac{2}{n}+\frac{47}{3n^2}
    + o\left(\frac{1}{n^2}\right)\right)\nonumber\\
 =&\, \left(\frac{C_1C_3-C_2^2}{C_1^2}\right)n^2
    +\left(\frac{C_3}{3C_1}+\frac{C_2}{C_1}-2\frac{C_2^2}{C_1^2}\right)n
     +C_4+o(1),\label{asyp-sigma_n^2}
\end{align}
where $C_4$ is a finite real number.
To prove $\sigma_n^2\rightarrow +\infty$, it suffices to show that
\begin{align}\label{re:coe2>0}
 C_1C_3-C_2^2>0,
\end{align}
or
\begin{align}\label{re:coe2=0&coe1>0}
  C_1C_3-C_2^2=0 \quad \textrm{ and }\quad
  C_1C_3+3C_1C_2-6C_2^2>0.
\end{align}

As will be seen, \eqref{re:coe2=0&coe1>0} holds true. In order to prove \eqref{re:coe2=0&coe1>0}, we need relations among $C_1$, $C_2$ and $C_3$. By \eqref{eq:asy-fn1} and \eqref{eq:asy-fn'1}, we see that
\begin{align}\label{eq:lim-C2/C1}
 \lim_{n\rightarrow +\infty} \frac{f_n'(1)}{nf_n(1)}=\frac{C_2}{C_1}.
\end{align}
And by \eqref{eq:asy-fn1} and \eqref{eq:asy-fn''1}, we have
\begin{align}\label{eq:lim-C3/C1}
 \lim_{n\rightarrow +\infty} \frac{f_n''(1)}{n^2 f_n(1)}=\frac{C_3}{C_1}.
\end{align}

To find the desired relations of $C_1$, $C_2$ and $C_3$, we need tractable expressions of ${f_n'(1)}/{f_n(1)}$ and ${f_n''(1)}/{f_n(1)}$. Note that $f_n(1)$, $f_n'(1)$ and $f_n''(1)$ can not be represented in a closed form. So, suitable recurrence relations regarding $f_n(x)$, $f_n'(x)$ and $f_n''(x)$ play a key role in the remainder of our proof. 
For this purpose, we employ the recurrences stated in Lemma \ref{lem:rec-fn''2}, which gives a system of linear equations in the variables $f_n'(x)$ and $f_n''(x)$. That is,
\begin{equation}\label{rec:PBnx3-4-f}
\left\{
\begin{aligned}
  f_{n+1}(x)
=&\  a_{n}(x) f_{n}(x)
   +b_{n}(x) f_{n}'(x)
   +c_{n}(x) f_{n}''(x),\\
  f_{n}(x)
=&\  \widetilde{a}_{n}(x) f_{n+1}'(x)
   +\widetilde{b}_{n}(x) f_{n}'(x)
   +\widetilde{c}_{n}(x) f_{n}''(x),
\end{aligned}
\right.
\end{equation}
where $a_n(x)$, $b_n(x)$, $c_n(x)$, $\widetilde{a}_n(x)$, $\widetilde{b}_n(x)$ and $\widetilde{c}_n(x)$ are given in Lemma \ref{lem:rec-fn''2}.

Clearly, $b_n(1)$, $c_n(1)$, $\widetilde{b}_n(1)$ and $\widetilde{c}_n(1)$ are nonzero.
Setting $x=1$ in \eqref{rec:PBnx3-4-f}, then dividing by $f_n(1)$ gives
\begin{equation}\label{rec:PBnx3-1}
\left\{
\begin{aligned}
  \frac{f_{n+1}(1)}{f_n(1)}
=&\  a_{n}(1)
   +b_{n}(1) \frac{f_{n}'(1)}{f_{n}(1)}
   +c_{n}(1) \frac{f_{n}''(1)}{f_{n}(1)},\\
  1
=&\  \widetilde{a}_{n}(1) \frac{f_{n+1}'(1)}{f_{n}'(1)}\frac{f_{n}'(1)}{f_{n}(1)}
   +\widetilde{b}_{n}(1) \frac{f_{n}'(1)}{f_{n}(1)}
   +\widetilde{c}_{n}(1) \frac{f_{n}''(1)}{f_{n}(1)}.
\end{aligned}
\right.
\end{equation}
Solving \eqref{rec:PBnx3-1}, we get
\begin{equation}\label{resolv-musigma}
\left\{
\begin{aligned}
  \frac{f_{n}'(1)}{f_{n}(1)}
=&\ \frac{\displaystyle 1-\frac{\widetilde{c}_n(1)}{c_n(1)}\left(\frac{f_{n+1}(1)}{f_n(1)}-a_n(1)\right)}
    {\displaystyle \widetilde{a}_n(1)\frac{f_{n+1}'(1)}{f_{n}'(1)}+\widetilde{b}_n(1)
     -\frac{\widetilde{c}_n(1)b_n(1)}{c_n(1)}},\\[5pt]
  \frac{f_{n}''(1)}{f_{n}(1)}
=&\ \frac{\displaystyle \frac{f_{n+1}(1)}{f_n(1)}
  - a_n(1)-b_n(1)\frac{f_{n}'(1)}{f_{n}(1)}}{c_{n}(1)}.
\end{aligned}
\right.
\end{equation}

We are now able to compute the left-hand side of \eqref{eq:lim-C2/C1} with the expression of ${f_{n}'(1)}/{f_{n}(1)}$ given by \eqref{resolv-musigma} and the limitations of ${f_{n+1}(1)}/{f_n(1)}$ and ${f_{n+1}'(1)}/{f_n'(1)}$ at $+\infty$ given in Lemma \ref{lme:4.2}.
It follows from \eqref{eq:lim-C2/C1}, \eqref{resolv-musigma} and \eqref{eq:lim-fn1n&fn1'n'} that
\begin{align}\label{eq:val-C2/C1}
 \frac{C_2}{C_1}=\lim_{n\rightarrow +\infty} \frac{f_n'(1)}{nf_n(1)}
=&\, \lim_{n\rightarrow +\infty}\frac{\displaystyle
    1-\frac{\widetilde{c}_n(1)}{c_n(1)}\left(\frac{f_{n+1}(1)}{f_n(1)}-a_n(1)\right)}
    {\displaystyle n\left(\widetilde{a}_n(1)\frac{f_{n+1}'(1)}{f_{n}'(1)}+\widetilde{b}_n(1)
     -\frac{\widetilde{c}_n(1)b_n(1)}{c_n(1)}\right)}\nonumber\\
=&\, \lim_{n\rightarrow +\infty}\frac{\displaystyle
   1-\frac{\widetilde{c}_n(1)}{c_n(1)}(8-a_n(1))}
    {\displaystyle 8n\widetilde{a}_n(1)+n\widetilde{b}_n(1)
     -n\frac{\widetilde{c}_n(1)b_n(1)}{c_n(1)}}\nonumber\\
=&\, \lim_{n\rightarrow +\infty}
     \frac{9n^2+41n+48}{6n(3n+8)}\nonumber\\
=&\, \frac{1}{2},
\end{align}
and hence,
\begin{align}\label{eq:fn'1/fn1-eqo}
 \mu_n=\frac{f_n'(1)}{f_n(1)}
 \sim \frac{n}{2}.
\end{align}

It follows from \eqref{eq:lim-C3/C1}, \eqref{resolv-musigma}, \eqref{eq:lim-fn1n&fn1'n'} and \eqref{eq:fn'1/fn1-eqo} that
\begin{align}\label{eq:val-C3/C1}
 \frac{C_3}{C_1}=\lim_{n\rightarrow +\infty} \frac{f_n''(1)}{n^2 f_n(1)}
=&\, \lim_{n\rightarrow +\infty} \frac{\displaystyle \frac{f_{n+1}(1)}{f_n(1)}
  - a_n(1)-b_n(1)\frac{f_{n}'(1)}{f_{n}(1)}}{n^2 c_{n}(1)}\nonumber\\
=&\, \lim_{n\rightarrow +\infty}
      \frac{\displaystyle 8-11+\frac{12n}{(n+3)(n+4)}\cdot
      \frac{n}{2}}{12}\nonumber\\
=&\, \frac{1}{4}.
\end{align}

By \eqref{eq:val-C2/C1} and \eqref{eq:val-C3/C1},
\begin{align}\label{eq:C123}
C_1=2C_2=4C_3.
\end{align}
Clearly,
$$
C_1 C_3=C_2^2,
$$
and
\begin{align*}
 C_1C_3+3C_1C_2-6C_2^2=C_2^2+6C_2^2-6C_2^2=C_2^2>0.
\end{align*}
which leads to \eqref{re:coe2=0&coe1>0}. Moreover, by \eqref{asyp-sigma_n^2}, \eqref{re:coe2=0&coe1>0} and \eqref{eq:C123},
\begin{align*}
 \sigma_n^2
\sim \left(\frac{C_3}{3C_1}+\frac{C_2}{C_1}-2\frac{C_2^2}{C_1^2}\right)n
= \frac{n}{12}.
\end{align*}
This completes the proof.
\end{proof}

\end{document}